\documentclass{article}

\usepackage{amsmath,
            amssymb,
            latexsym}

\usepackage[pdftex]{graphicx} 
\usepackage{epstopdf} 

\newtheorem{theorem}{Theorem}[section]
\newtheorem{lemma}[theorem]{Lemma}

\newtheorem{proposition}[theorem]{Proposition}
\newtheorem{definition}[theorem]{Definition}

\newenvironment{proof}{\noindent\textsc{Proof: }}
{\hspace{\stretch{1}}$\Box$\medskip}

\begin{document}

\title{Topological representation of matroids\\ from diagrams of spaces}

\author{
  Alexander Engstr\"om\footnote{The author is Miller Research Fellow 2009-2012 at UC Berkeley, and gratefully acknowledges support from the Adolph C. and Mary Sprague Miller Institute for Basic Research in Science.}
  \\
Department of Mathematics \\
UC Berkeley \\ 
\texttt{alex@math.berkeley.edu}
}

\date\today

\maketitle

\begin{abstract}
Swartz proved that any matroid can be realized as the intersection lattice of an arrangement of codimension one homotopy spheres on a sphere. This was an unexpected extension from the oriented matroid case, but unfortunately the construction is  not explicit. Anderson later provided an explicit construction, but had to use cell complexes of high dimensions that are homotopy equivalent to lower dimensional spheres.

Using diagrams of spaces we give an explicit construction of arrangements in the right dimensions. Swartz asked if it is possible to arrange spheres of codimension two, and we provide a construction for any codimension. We also show that all matroids, and not only tropical oriented matroids, have a pseudo-tropical representation.

We determine the homotopy type of all the constructed arrangements. 
\end{abstract}

\section{Introduction}

It is a classic result about oriented matroids, that they can be realized as the intersection lattice of an arrangement of pseudo-spheres on a sphere \cite{BLVSWZ}. In a nice paper Swartz \cite{S} proved the unexpected result that \emph{any} matroid can be realized as the intersection lattice of homotopy codimension one spheres on a sphere. We give a new proof of this result, but with an explicit construction using diagrams of spaces. Anderson \cite{A} also gave a new construction, but her representation used spheres of higher dimension than necessary.

Our construction is much more general than only about codimension one spheres on spheres. An $n$-dimensional sphere can be described as the $(n+1)$--fold join of two disjoint vertices. An arrangement of codimension one spheres on a sphere, can be described as an arrangement of $X^{\ast n}$ complexes on  $X^{\ast (n+1)}$, where $X^{\ast n}$ denotes the $n$--fold join of $X$, and $X$ is two disjoint vertices. The first part of the Topological Representation Theorem~\ref{theorem:rep} states that any $X$ can be used to get a matroid representation.

Swartz asked if it is possible represent a matroid with codimension two spheres on a sphere. If we choose $X$ to be a one-sphere, that is achieved. And if $X$ is three disjoint points we get a tropical representation. Other $X$ than this is also discussed later.
The second part of the Topological Representation Theorem~\ref{theorem:rep} states that the union of all $X^{\ast n}$ complexes on  $X^{\ast (n+1)}$ is homotopy equivalent to
\[ \bigvee_{ p \in M\setminus \hat{0} } \left(    X^{ \ast (d-\textrm{rank}(p))}   \ast    \bigvee^{ | \mu(\hat{0},p) | }  S^{\textrm{rank}(p)-2}     \right) \]
if $M$ is the rank $d$ matroid represented.

In Section 2 we describe methods from discrete Morse theory and the theory of diagrams of spaces. These methods are then used in Section 3 to define the arrangements and prove the Topological Representation Theorem. In the last section we give an explicit example.

\section{Tools from combinatorics and topology}

In this part we describe the two main tools used in this paper, discrete Morse theory and diagrams of spaces. For a basic introduction to combinatorial topology, we recommend Bj\"orner's survey \cite{B}. For matroids, and in particular oriented ones, see \cite{BLVSWZ}.

\subsection{Discrete Morse theory}

Discrete Morse theory was introduced by Forman \cite{F} as a method to find cell structures for topological spaces. The goal is to replace a simplicial complex or a regular CW complex by a CW complex of the same homotopy type but with fewer cells. A CW complex is \emph{regular} if all of its cells are homeomorphic to euclidean cells, and that is exactly when the topology of the CW complex is captured by the combinatorial data compiled in its face poset.  The \emph{face poset} $\mathcal{F}(X)$ of a regular CW complex $X$ is a poset with the cells of $X$ as elements and inclusion as order relation.  The \emph{Hasse diagram} of the faces of $X$ is the directed graph with the cells of $X$ as vertex labels and arrows $\sigma \rightarrow \tau$ for every codimension one subcell $\tau$ of $\sigma$. In the context of homological algebra or topology the arrows in diagrams usually encode cofibrations or inclusions, and are turned the other way around. Working with discrete Morse theory one should rather think of the arrows from a cell as describing where the boundary map sends things.

A \emph{matching $M$ on $X$} is a set of arrows in the Hasse diagram of $X$ with no common vertices. The matching $M$ is \emph{acyclic}, and sometimes called a Morse matching, if the Hasse diagram is still acyclic if we swap the directions of all arrows in $M$. The cells of $X$ adjacent to arrows of $M$ are referred to as the cells of $M$, and the cells of $X$ not in $M$ are the \emph{critical cells}. The critical cells in discrete Morse theory has the same function as critical cells in ordinary Morse theory \cite{F,Mo}. Discrete Morse theory was originally formulated with Morse functions \cite{F}, but Chari's matchings \cite{Ch} are more suitable in this paper.

\begin{theorem}[Main theorem of discrete Morse theory]\label{theorem:mainDiscreteMorseTheory}
If $X$ is a regular CW complex with an acyclic matching $M$ (giving at least one critical vertex), then there exists a CW complex $Y$ that is homotopy equivalent to $X$, and the number of $d$-dimensional cells of $Y$ equals the number of $d$-dimensional critical cells of $X$ for every $d$.
\end{theorem}

A regular CW complex is \emph{collapsible} if there is a matching $M$ on $X$ with only one critical cell, and that is a vertex. Collapsible complexes were studied long before discrete Morse theory, and they are in particular contractible. This elementary lemma is almost a corollary of the main theorem of discrete Morse theory.

\begin{lemma}\label{lemma:elementaryDiscreteMorseTheory}
Let $X$ be a regular CW complex with an acyclic matching $M$ and a critical vertex $v$, such that the cells of $M$ together with $v$ is a subcomplex $Z$ of $X$. Then there is a CW complex $Y$ and a homotopy equivalence $f:X \rightarrow Y$ such that $Z$ is mapped onto a vertex of $Y$.
\end{lemma}

\begin{proof}
The subcomplex $Z$ is contractible since it is collapsible. Set $Y=X/Z$ and $Z$ is mapped onto the point $Z/Z$.
\end{proof}

How to find  acyclic matchings? One standard trick is locating a cone $v \ast Y \subset X$ and then $\{ v\ast \tau \rightarrow \tau \mid \tau\in Y \}$ is an acyclic matching on $X$. We want to patch simple acyclic matchings together in a systematic way, and Hersh \cite{He} and Jonsson \cite{J} independently found the right way to do that. First note that there is an obvious generalization of the concepts Hasse diagram, matching, and acyclic matching to other finite posets than face posets.

\begin{lemma}\label{lemma:composeAcyclic}
If $f:P\rightarrow Q$ is a poset map and $M_q$ is an acyclic matching on each $f^{-1}(q) \subset P$, then $M=\cup_{q\in Q} M_q$ is an acyclic matching on $P$.
\end{lemma}

\subsection{Diagrams of spaces}

A diagram of spaces is a functor from a small category into the category of topological spaces. The most comprehensive reference geard towards applictions in combinatorics is \cite{WZZ} following up on \cite{ZZ}. Most of these methods were worked out in the context of category theory \cite{DS,M,V}. We need a quite limited part of the machinery: our small category is represented by finite posets and the topological spaces in the diagram are CW complexes.

\begin{definition}
A \emph{ $P$-diagram of spaces $\mathcal{D}$} is the following data
\begin{itemize}
\item a finite poset $P$,
\item a CW complex $D_p$ for every $p\in P$,
%
%
%
%
\item a continuous map $d_{pq}: D_p \rightarrow D_q$ for every pair $p\geq q$ of $P$,
\end{itemize}
satisfying that $d_{qr}(d_{pq}(x))=d_{pr}(x)$ for every tripple $p\geq q \geq r$ of $P$ and $x\in D_p$.
\end{definition}

For categories without an obvious associated homotopy theory one would realize the diagrams using colimits, but the category of topological spaces is the archetypical category for working with homotopy colimits. The general the definition of the homotopy colimit (Section 2.2, \cite{WZZ}) is more complicated than this, but this is up to homotopy the same in our setting (Proposition 4.1, \cite{WZZ}).

\begin{definition}[Homotopy colimit representation]\label{def:homotopyColimitConstruction}
Let $\mathcal{D}$ be a $P$-diagram. Define $\mathcal{J}$ as the join of all spaces in the diagram realized by embedding them in skew affine subspaces. The points of $\mathcal{J}$ are parametrized as
\[  \left\{ \sum_{p\in P} t_px_p \left|  x_p\in D(p) \textrm{ and } 0\leq t_p\leq 1 \textrm{ for all }p\in P, \textrm{ and } \sum_{p\in P} t_p =1  \right.  \right\}. \]
The \emph{homotopy colimit of $\mathcal{D}$} is
%
%
%
%
\[ \emph{\textrm{hocolim }} \mathcal{D} = \{ t_0x_0+\cdots+ t_mx_m\in \mathcal{J} \mid x_i \in D_{p_{i}}, p_0\leq \cdots \leq p_m, d_{p_{i+1}p_i}(x_{i+1})=x_i  \}. \]
\end{definition}

\begin{lemma}[Homotopy lemma]\label{lemma:homotopyLemma}
Let $\mathcal{D}$ and $\mathcal{E}$ be $P$-diagrams  with homotopy equivalences  $\alpha_p:D_p \rightarrow E_p$ for every $p\in P$ such that the diagram
%
%
%
%
%
%
%
%
%
\[    \begin{array}{ccccc} & & \alpha_p & \\  & D_p & { \rightarrow } & E_p  \\
 d_{pq} & \downarrow  && \downarrow & e_{pq}    \\
 & D_q & { \rightarrow } & E_q  \\  & & \alpha_q & \\ 
  \end{array}   \]
  commutes for all $p>q$. Then $\alpha$ induces a homotopy equivalence
  \[ \hat{\alpha} :  \textrm{\emph{hocolim}} \,\,\mathcal{D} \rightarrow  \textrm{\emph{hocolim}} \,\,\mathcal{E}. \] 
\end{lemma}

In the (combinatorial topologists) perfect world all simplicial complexes would be shellable and bouquets of spheres. Many versions of this wedge lemma have been used in situation almost that good, when you have a diagram with null homotopic maps between spheres and balls.

\begin{lemma}[Wedge lemma]\label{lemma:wedgeLemma}
Let $P$ be a poset with a maximal element $\hat{1}$. Let $\mathcal{D}$ be a $P$-diagram with
%
%
%
%
%
%
%
%
pointed spaces $b_p \in D_p$ for every $p\in P \setminus \hat{1}$ such that $d_{pq}$ is the constant map $x \mapsto b_q$ for every pair $p>q$. Then
\[ \textrm{\emph{hocolim}} \,\,\mathcal{D} \simeq \bigvee_{p\in P} \left( \Delta(P_{<p}) \ast D_p \right) \]
\end{lemma}

So far we haven't defined any diagrams given CW complexes.
\begin{definition}
Let $X$ be a CW complex with subcomplexes $X_1, X_2, \ldots, X_n$ such that
\begin{itemize}
\item the union of the subcomplexes is $X$,
\item and the intersection of any set of subcomplexes $X_{i_1},\ldots X_{i_m}$ is a subcomplex of $X$,
\end{itemize}
then the \emph{intersection diagram} is a diagram whose elements are the intersections of subcomplexes and the maps are the inclusion morphisms.
\end{definition}
In the previous definition we could have skipped the condition that the intersection of subcomplexes is a subcomplex of $X$ if we had assumed that $X$ was a regular CW complex or a simplicial complex. The intersection diagram is over a poset with a maximal element (often denoted $\hat{1}$) whose space is the intersection of all the subcomplexes.

\begin{lemma}[Projection lemma]\label{lemma:projectionLemma}
Let $\mathcal{D}$ be the intersection diagram for $X$. Then there is a homotopy equivalence
\[  X \simeq  \textrm{\emph{hocolim}} \,\,\mathcal{D}.  \] 
\end{lemma}

\section{Topological representations of a matroid}

This is the main part of the paper, and in here we define what topological representation of matroids are, and give an explicit way to find many of them. We also use the tools introduced in the previous section to determine the homotopy type of the spaces representing the matroids.

\begin{definition}
A rank $r$ matroid $M$ is a collection of subsets of a finite set $E$ with a poset structure by inclusion, such that
\begin{itemize}
\item The poset $M$ is graded of rank $r$.
\item The poset $M$ is a geometric lattice:
\begin{itemize}
\item The join $p \vee q$ and meet $p \wedge q$ is defined for all $p,q \in M$.
\item Every element is the join of atoms.
\item Semimodularity: $\textrm{rank}(p)+\textrm{rank}(q) \geq \textrm{rank}(p \wedge q) + \textrm{rank}(p \vee q)$ for  all $p,q \in M$.
\end{itemize}
\item The meet is the intersection: $p \wedge q = p \cap q $ for all $p,q \in M$.
\item The join of all atoms is $E$.
\end{itemize}
\end{definition}

The \emph{flats of a matroid $M$} are all sets
\[ F= \left\{  A \subseteq \textrm{atoms}(M) \left| \bigvee_{a \in A} a \not \geq a' \textrm{ for all } a' \in \textrm{atoms}(M)\setminus A  \right.  \right\}.  \]
With inclusion order, the flats $F$ is a rank $r$ matroid isomorphic to $M$ by:
\[ \begin{array}{ccc}  f: M \rightarrow F  &&  f(p)=\{ a \textrm{ atom of } M \mid a \leq p \}  \\
  g: F \rightarrow M  &&  \displaystyle g(A)= \bigvee_{a \in A} a  \\
\end{array}
\]

\begin{definition}\label{def:rep}
An $X$-\emph{arrangement} is a CW complex $Y$ and a finite set $\mathbf{A}$ of subcomplexes of $Y$ such that:
\begin{itemize}
\item[1.] The complex $Y$ is homotopy equivalent to $X^{\ast d}$ for some $d$, and $\dim (Y) = \dim (X^{\ast d})$.
\item[2.] Each complex $A$ in $\mathbf{A}$ is homotopy equivalent to $X^{\ast (d-1)}$ and $\dim (A) = \dim (X^{\ast (d-1)})$.
\item[3.] Each intersection $B$ of complexes in $\mathbf{A}$ is homotopy equivalent to some  $X^{\ast e}$, and $\dim (B) = \dim (X^{\ast e})$.
\item[4.] If there is a free group action of $\Gamma$ on $X$, then it induces a free $\Gamma$-action on $Y$ and every intersection of complexes in $\mathbf{A}$.
\item[5.] If $B\simeq X^{\ast e}$ is an intersection of complexes in $\mathbf{A}$, the complex $A$ is in $\mathbf{A}$, and 
%
%
%
$A \not \supseteq B$, then
$A\cap B \simeq X^{\ast (e-1)}$.
\end{itemize}
\end{definition}

A subset $\{ A_1, A_2, \ldots A_n\}$ of $\mathbf{A}$ is a \emph{flat in the arrangement} if for every subcomplex $B$ in $\mathbf{A}\setminus \{A_1,A_2,\ldots, A_n\}$
\[ \bigcap_{i=1}^n A_n \not \subseteq B. \]

\begin{proposition}\label{prop:isMatroid}
Let $(Y,\mathbf{A})$ be an $X$-arrangement, then the flats of $\mathbf{A}$ is a matroid.
\end{proposition}

The proof of Proposition~\ref{prop:isMatroid} is tedious but elementary, and is left to the reader. Alternatively replace every "$S^{e-1}$" by "$X^{\ast e}$" in the proof of Proposition 2.3 in \cite{A}.

Swartz \cite{S} found that any matroid can be represented with codimension one spheres on a sphere. In our setting this corresponds to the choice of $X$ being two isolated vertices, and
\[ X^{\ast d} = (\cdot \, \cdot)^{\ast d} = S^{d-1}. \]
This is exactly setup of intersecting a real centralized hyperplane arrangement with a sphere to get an arrangement of spheres. In Swartz' \cite{S} paper it is asked for a more general setup that also would cover the representation of matroids with codimension two spheres on a sphere. This corresponds to complex hyperplane arrangements, and we get these representations by using $X=S^{1}$,
\[ X^{\ast d} = (S^1)^{\ast d} = S^{2d-1}. \]
The uncommon, but still studied, quaternion hyperplane arrangements \cite{EPS} have their corresponding matroid representation with $X=S^3$.

A tropical oriented matroid is an arrangement of tropical hyperplanes with the cells in its complement labeled with three covectors $(1,2,3)$ per hyperplane instead of two covectors $(-,+)$ as for the ordinary oriented matroid \cite{AD}. The intersection of a one dimensional tropical hyperplane and $S^1$ is three points, and $(\cdot \, \cdot \, \cdot)^{\ast d} - (\cdot \, \cdot \, \cdot)^{\ast (d-1)}\ast \emptyset$ is homotopic to three points that could represent the covectors $(1,2,3)$. In some sense $X=(\cdot \, \cdot \, \cdot)$ gives us a tropical representation of matroids. A tropical hyperplane is a codimension one skeleton of a fan with three maximal cells. Using fans with $k$ maximal cells one could define a generalized oriented matroid with $k$ covectors per hyperplane. For $k=2$ this is the ordinary oriented matroid, and for $k=3$ it is the tropical matroid. Setting $X$ to $k$ isolated vertices we get a $k$-fan representation of matroids.

This lemma is used to determine the homotopy type of our arrangements.
 
\begin{lemma}\label{lemma:alexWedgeLemma}
Let $\mathcal{D}$ be a $P$-diagram of regular CW complexes where all maps are inclusion morphisms. Suppose that there are acyclic matchings $M_q$ on every $D_q$ for $q\in P\setminus \hat{1}$ such that
\begin{itemize}
\item for some vertex $b_q$ of $D_q$, the set of cells of $M_q$ together with the vertex $b_q$ is a collapsible subcomplex of $D_q$; 
\item every cell of $D_p\setminus \{b_q\}$ is in $M_q$, if $q$ covers $p$ in $P$;
\end{itemize}
then we have a homotopy equivalence
\[ X  \simeq \bigvee_{ p \in P} \left( \Delta(P_{<p }) \ast D_p \right) \]
\end{lemma}

\begin{proof}

We construct a $P$-diagram $\mathcal{E}$ whose homotopy colimit is homotopy equivalent with the homotopy colimit of $\mathcal{D}$, but the maps in $\mathcal{E}$ are null homotopic and we can then use the wedge lemma.

For $q< \hat{1}$ define $E_q=D_q/Z_q$ where $Z_q$ is the collapsible subcomplex of $D_p$ whose cells are the cells of $M_q$ together with $b_q$. According to Lemma~\ref{lemma:elementaryDiscreteMorseTheory} this gives a homotopy equivalence $\alpha_q : D_q \rightarrow E_q$ with $\alpha_q(Z_q)=c_q$ for some point $c_q$ of $E_q$. Replace every map $e_{pq}$ in $\mathcal{D}$ by the constant map $e_{pq}(x)=c_q$ in  $\mathcal{E}$. Also set $E_{\hat{1}}=D_{\hat{1}}$ and $\alpha_{\hat{1}}(x)=x$. The data for $\mathcal{E}$ is a $P$-diagram of spaces since the maps are coherent: $e_{qr}(e_{pq}(x))=e_{pr}(x)=c_r$ for all $x$ in $E_p$. The diagram
\[    \begin{array}{ccccc} & & \alpha_p & \\  & D_p & { \rightarrow } & E_p  \\
 d_{pq} & \downarrow  && \downarrow & e_{pq}    \\
 & D_q & { \rightarrow } & E_q  \\  & & \alpha_q & \\ 
  \end{array}   \]
  commutes for all $p>q$ since $\alpha_q(d_{pq}(D_p)) \subseteq \alpha_q( Z_q ) = \{ c_q \}$ and $e_{pq}(\alpha_p(D_p))\subseteq e_{pq}(E_p)= \{ c_q \}$. By the homotopy lemma (Lemma~\ref{lemma:homotopyLemma})
\[ \textrm{hocolim }\mathcal{D} \simeq  \textrm{hocolim }\mathcal{E}, \]
and by the wedge lemma (Lemma~\ref{lemma:wedgeLemma})
\[   \textrm{hocolim }\mathcal{E}  \simeq \bigvee_{ p \in P} \left( \Delta(P_{<p }) \ast E_p \right) \simeq \bigvee_{ p \in P} \left( \Delta(P_{<p }) \ast D_p \right)  \]
where the last homotopy equivalence follows from that $E_p \simeq D_p$ by construction.
%
%

\end{proof}

In the applications of the previous lemma in this paper, the $P$-diagram $\mathcal{D}$ is an intersection diagram and the lemma is used together with the projection lemma. A special case of the previous lemma is when one of the subcomplexes of $X$ used to construct the intersection diagram is $X$. Then $\Delta(P_{<p})$ is a cone and $\Delta(P_{<p }) \ast D_p$ is contractible for all $p \neq \hat{0}$. The only thing left to wedge is $\Delta(P_{< \hat{0} }) \ast D_{\hat{0}}= \emptyset \ast X=X$ and we get $X \simeq X$.

The next lemma explains how to construct many topological spaces with null homotopic maps between them, to be used in the previous lemma.

\begin{lemma}\label{lemma:joinUp}
If $X_1, X_2, \ldots, X_n$ are regular CW complexes then there exists an acyclic matching $M$ on $\mathcal{F}(X_1 \ast X_2 \ast \cdots \ast X_n)$ and a critical vertex $b$ of that matching, such that
\begin{itemize}
\item the cells of $M$ together with $b$ is a collapsible subcomplex $Z$ of $X_1 \ast X_2 \ast \cdots \ast X_n$, and
\item the complex $ X_1 \ast \cdots \ast X_{i-1} \ast \emptyset \ast  X_{i+1} \ast \cdots X_n$ is a subcomplex of $Z$ for all $i=1,2,\ldots,n$.
\end{itemize}
\end{lemma}

\begin{proof}

The face poset of $X_1 \ast X_2 \ast \cdots \ast X_n$ is isomorphic to
\[ P = ( \mathcal{F}(X_1) \cup \hat{0} ) \times ( \mathcal{F}(X_2) \cup \hat{0} ) \times \cdots \times ( \mathcal{F}(X_n) \cup \hat{0} ) \setminus \{\hat{0}\} \times  \{\hat{0}\} \times
 \cdots  \times \{\hat{0}\}. 
 \]
Pick some arbitrary vertices $x_i$ of $X_i$. Define a poset map
\[ g: P \rightarrow \{ 0 < 1 \}^n \]
coordinate wise by
\[ g( \ldots,\sigma_i,\ldots )_i = \left\{ \begin{array}{cl} 1 & \sigma \not\in \{ \hat{0} , x_i \} \\ 0 &   \sigma \in \{ \hat{0} , x_i \}    \end{array}  \right.\]
We don't care about creating a matching on the pre-image of $(1,1,1, \ldots, 1)$. For any other $s \in    \{ 0 < 1 \}^n$ let $l$ be the smallest number such that $s_l=0$. There is a complete acyclic matching on $g^{-1}(s)$ by matching $\hat{0}$ with $x_l$. Well, the matching is not complete on $g^{-1}(0,0,\ldots,0)$ since we try to match the nonexisting $(\hat{0},\hat{0},\ldots,\hat{0})$ with $(x_1,\hat{0},\ldots,\hat{0})$. So, $(x_1,\hat{0},\ldots,\hat{0})$ is unmatched and set $b=(x_1,\hat{0},\ldots,\hat{0})$. By Lemma~\ref{lemma:composeAcyclic} this matching is acyclic together, and together with $b$ its cells gives a collapsible subcomplex $Z$.

Any cell $\sigma $ of  $ X_1 \ast X_2 \ast \cdots \ast X_n $ not in $Z$ satisfy  $g( \ldots,\sigma_i,\ldots )_i = 1$ and $\sigma \not\in X_1 \ast \cdots \ast X_{i-1} \ast \emptyset \ast  X_{i+1} \ast \cdots X_n$  for all $i=1,2,\ldots,n$.
\end{proof}

This concept is used in Theorem~\ref{theorem:rep}, the Representation Theorem of Matroids.
\begin{definition}
A \emph{rank- and order-reversing} poset map $\ell$ between the two ranked posets $P$ and $Q$ of the same rank, is a map that satisfies $\textrm{rank}(p)+\textrm{rank}(\ell(p))=\textrm{rank}(P)$ for all $p$ in $P$.
\end{definition}
For any matroid $M$, we make use of rank- and order-reversing maps from $M$ to the boolean lattice on $\{ 1,2, \ldots, \textrm{rank}(M) \}$. There are many such maps, and the easy choice is
\[ \ell(p) = \{ 1, 2, \ldots, \textrm{rank}(M)-\textrm{rank}(p) \} \]
where $B$ is a boolean lattice on $\{1,2, \ldots, \textrm{rank}(M) \}.$ To get other representation of $l$, which are anyway homotopy equivalent, one can use other $\ell$. The following set of $\ell$-maps is implicit in \cite{A}: Let $\hat{0}=p_1<p_2<\cdots<p_{\textrm{rank}(M)+1}=\hat{1}$ be a maximal flag of $M$ and set
\[ \ell(p) = \{ 1\leq i \leq  \textrm{rank}(M) \mid p_i \geq p \}. \]

\begin{theorem}[The Representation Theorem of Matroids]\label{theorem:rep}
Let $M$ be a rank $r$ matroid, and
%
%
 let $\ell$ be a 
 rank- and order-reversing poset map from $M$ to a boolean lattice on \{1,2, \ldots, r\}. Let $X$ be a finite regular CW complex and define
 \[ D_p = \ast_{i=1}^r  \left\{ \begin{array}{cl} X & \textrm{if $i \in \ell(p)$} \\ \emptyset & \textrm{if $i \not\in \ell(p)$}    \end{array} \right. \]
 to get an $M$-diagram $\mathcal{D}$ with inclusion morphisms.
 
 Then
 \[ (Y,\mathbf{A})=  ( \textrm{hocolim }\mathcal{D}, \{ \textrm{hocolim }\mathcal{D}_{\geq a} \mid a \textrm{ is an atom of } M \} ) \]
 is an $X$--arrangement of $M$ and
 \[ \bigcup_{A \in \mathbf{A}} A \simeq \bigvee_{ p \in M\setminus \hat{0} } \left(    X^{ \ast (d-\textrm{rank}(p))}   \ast    \bigvee^{ | \mu(\hat{0},p) | }  S^{rank(p)-2}     \right)     \]
\end{theorem}

\begin{proof}

We need some more notation. As for any poset, $M_{\geq q}$ is the part of $M$ great or equal to $q$, and $M_{>q}$ is the part of $M$ greater than $q$. In the theorem statement a $M$-diagram of spaces $\mathcal{D}$ was defined. The smaller diagrams of spaces $\mathcal{D}_{\geq q}$ and $\mathcal{D}_{>q}$ use the same topological spaces and maps, but only the smaller posets $M_{\geq q}$, and $M_{>q}$. Even more general, if $P$ is a subposet of $M$, let $\mathcal{D}[P]$ be the diagram of spaces on $P$ defined by the induced subposet $P$ and $M$-diagram $\mathcal{D}$.

First we prove that it is an $X$--arrangement of $M$. The subcomplexes of $Y$ in $\mathbf{A}$, and there intersections, are on the form
\[  \textrm{hocolim }\mathcal{D}_{\geq q}\]
as shown later. From the wedge lemma (Lemma \ref{lemma:wedgeLemma}) we get that
\[  \textrm{hocolim }\mathcal{D}_{\geq q}   \simeq \bigvee_{p\in M_{\geq q}} \left( \Delta(M_{<p \cap M_{\geq q}}) \ast D_p \right) = D_q = X^{\ast (\textrm{rank}(M)-\textrm{rank}(q))} \]
since $ \Delta(M_{<p \cap M_{\geq q}}) $ is a cone if $p>q$ and the empty set if $p=q$.  Also note that the dimension of $\textrm{hocolim }\mathcal{D}_{\geq q}  $ is the same as for $X^{\ast (\textrm{rank}(M)-\textrm{rank}(q))}$ since the dimension of $\Delta(M_{<p \cap M_{\geq q}}) $ is $\textrm{rank}(p)-\textrm{rank}(q)-1$. Here is the verification of four out of five properties of $X$--arrangements, as stated in Definition \ref{def:rep}.
\begin{itemize}
\item[1.] The complex $Y$ is defined as $\textrm{hocolim }\mathcal{D}_{\geq \hat{0}}$, so $Y \simeq X^{\ast d}$ for $d=\textrm{rank}(M)$ and $\dim(Y)=\dim( X^{\ast d}).$
\item[2.]  Each complex $A$ in $\mathbf{A}$ is homotopy equivalent to $X^{\ast (d-1)}$ and $\dim(A)=\dim(X^{\ast (d-1)}$, since $A=\textrm{hocolim }\mathcal{D}_{\geq a}\simeq X^{\ast (\textrm{rank}(M)-\textrm{rank}(a))}=X^{\ast (d-1)}$ for some atom $a$ of $M$.
\item[3.] If $B$ is an intersection of $A_1,A_2,\ldots, A_k$ from $\mathbf{A}$, then
\[\begin{array}{rcl}
B & = &  \bigcap_{i=1}^k A_i \\
& = &   \bigcap_{i=1}^k \textrm{hocolim} \,\, \mathcal{D}_{\geq a_i} \\
& = &   \bigcap_{i=1}^k \textrm{hocolim} \,\, \mathcal{D}[M_{\geq a_i}] \\
& = &   \textrm{hocolim} \,\, \mathcal{D}\left[  \bigcap_{i=1}^k M_{\geq a_i}\right] \\
& = &   \textrm{hocolim} \,\, \mathcal{D}\left[ M_{\geq \bigvee_{i=1}^k a_i}\right] \\
& = &   \textrm{hocolim} \,\, \mathcal{D}_{\geq \bigvee_{i=1}^k a_i}\\
& \simeq & X^{\ast (d- \textrm{rank}( \bigvee_{i=1}^k a_i))} \\
\end{array}
\]
for some atoms $a_1, a_2, \ldots a_k$ of $M$, and $\dim(B)=\dim(X^ {\ast (d- \textrm{rank}( \bigvee_{i=1}^k a_i))} )$.
\item[5.] Now assume that $B \simeq X^{\ast e}$ is an intersection of complexes in $\mathbf{A}$, the complex $A$ is in $\mathbf{A}$, and 
%
%
%
%
%
%
$A \not \supseteq B$
We should show that $A\cap B \simeq X^{\ast (e-1)}$.

As shown in point 3 above, we know that $B$ is on the form
\[ B = \textrm{hocolim} \,\,  \mathcal{D}_{\geq \bigvee_{i=1}^k a_i} \simeq X^{\ast (d- \textrm{rank}( \bigvee_{i=1}^k a_i))} \]
for some atoms $a_1,a_2,\ldots, a_k$ of $M$. By definition
\[ A = \textrm{hocolim} \,\,   \mathcal{D}_{\geq a_0} \]
for some atom $a_0$ of $M$. From the assumption that 
%
%
%
%
%
%
$A \not \supseteq B$ we get that
\[ a_0 \not \leq  \bigvee_{i=1}^k a_i. \]
Now we reach the point of this proof where it is necessary to assume that $M$ is a matroid. Until here atomic lattice would have sufficed. Since $M$ is a matroid,
\[ \textrm{rank} \left(  \bigvee_{i=0}^k a_i \right) = 1 +  \textrm{rank} \left(  \bigvee_{i=1}^k a_i \right). \]
Using the description from point 3 again, we get that
\[\begin{array}{rcl}
A \cap B & = & \textrm{hocolim} \,\, \mathcal{D}_{\geq \bigvee_{i=0}^k a_i}\\
& \simeq &  X^{\ast (d- \textrm{rank}( \bigvee_{i=0}^k a_i))} \\
& = & X^{\ast (d- \textrm{rank}( \bigvee_{i=1}^k a_i)-1)} \\
& = & X^{\ast(e-1)}. \\
\end{array}\]
\end{itemize}
Point 4 in the definition of $X$--arrangement is not verified yet. It states that if $\Gamma$ is a free group action on $X$, then it induces a free $\Gamma$--action on $Y$ and every intersection of complexes in $\mathbf{A}$. If $p_1,p_2,\ldots, p_k$ are point of $X$ and $\gamma \in \Gamma$ then we act on $X ^{\ast k}$ by
\[ p_1 \ast p_2 \ast \cdots \ast p_k \mapsto \gamma(p_1) \ast \gamma(p_2)\ast \cdots \ast \gamma(p_k). \]
This is a free group action, and since all morphisms of the diagram $\mathcal{D}$ are inclusion morphisms, $\Gamma$ induces a free group action on $\textrm{hocolim}\,\, \mathcal{D}[P]$ for any subposet $P$ of $M$. Recall that any intersection of complexes in $\mathbf{A}$ is on the form $\textrm{hocolim}\,\, \mathcal{D}[P]$ by point 3 above. 

We have verified that $(Y,\mathbf{A})$ is an $X$--arrangement, and now continue with proving that
 \[ \bigcup_{A \in \mathbf{A}} A \simeq \bigvee_{ p \in M\setminus \hat{0} } \left(    X^{ \ast (d-\textrm{rank}(p))}   \ast    \bigvee^{ | \mu(\hat{0},p) | }  S^{rank(p)-2}     \right).  \]
By the definition of the homotopy colimit,
\[ \bigcup_{A \in \mathbf{A}} A = \bigcup_{a \textrm{ atom of }M}  \textrm{hocolim}\,\, \mathcal{D}_{\geq a} = \textrm{hocolim}\,\, \mathcal{D}_{>\hat{0}}.\]
To every element $p$ of $M_{>0}$ we have associated a space isomorphic to$X^{\ast (\textrm{rank}(M)-\textrm{rank}(p))}$. If $p$ covers $q$, then by definition $D_p \ast X$ is isomorphic to $D_q$ and the morphism in the diagram from $D_p$ into $D_q$ is an inclusion morphism. Using Lemma \ref{lemma:joinUp} we can deformation retract all spaces in the diagram to get null-homotopic maps, and then conclude by Lemma \ref{lemma:alexWedgeLemma} that
\[  \textrm{hocolim}\,\, \mathcal{D}_{>\hat{0}} \simeq \bigvee_{ p \in M_{>\hat{0}}} \left( \Delta(M_{>\hat{0}} \cap M_{<p }) \ast D_p \right). \]
According to Theorem 4.1 of \cite{Fo}, which is elaborated on in section 2 of \cite{B2}, the homotopy type of an interval in $M$ is
\[ \Delta( M_{>a} \cap M_{<b}) \simeq \bigvee^{ | \mu(a,b) | }  S^{rank(b)-rank(a)-2}.  \]
since it is a lattice. Here $\mu$ is the M\"obius function. Hence
\[ \Delta( M_{>\hat{0} } \cap M_{<b}   ) \simeq  \bigvee^{ | \mu(\hat{0},p) | }  S^{rank(p)-2}.  \]
By construction
\[D_p = X^{ \ast (d-\textrm{rank}(p))}, \]
and finally
\[ \bigcup_{A \in \mathbf{A}} A \simeq \bigvee_{ p \in M\setminus \hat{0} } \left(    X^{ \ast (d-\textrm{rank}(p))}   \ast    \bigvee^{ | \mu(\hat{0},p) | }  S^{rank(p)-2}     \right).  \]
\end{proof}

Given a $X$-representation $(Y,\mathbf{A})$ of a matroid $M$, how do we recover it? If we cover $\cup_{A\in \mathbf{A}} A$ by its subcomplexes in $\mathbf{A}$, then from point 5 in the first part of the proof of the Representation Theorem of Matroids~\ref{theorem:rep}, one see that the intersection diagram is a diagram on $M_{>\hat{0}}$.

\section{An explicit example: Representing the Fano matroid}
In this section we give an explicit example of how to use the Topological Representation Theorem \ref{theorem:rep}. The Fano matroid $M$ found by Whitney \cite{W} is a minimal non-oriented matroid \cite{BLVSWZ}.  In Figure~\ref{fig:fano} it is drawn with its circuits as curves. Let $X$ be two disjoint points to get a spherical representation of it. The map $\ell:M \rightarrow \{1,2,3,4\}$ is described in the diagram of spaces (Figure~\ref{fig:fanoLattice}) like this: If the $i$:th ball below a certain element of $p\in M$ is filled, then $i$ is in $\ell(p)$. For example,
\[ \ell(\{4,5,6,7\})=\{4\}, \,\,\,    \ell(\{4,7\})=\{1,4\},\textrm{and } \ell(\{7\})=\{1,2,4\} \]
In Figures~\ref{fig:fanoCyl1}-\ref{fig:fanoCyl4} some of the parts used to glue together the topological representation of $M$ is drawn.
\begin{figure}
\centering
\includegraphics{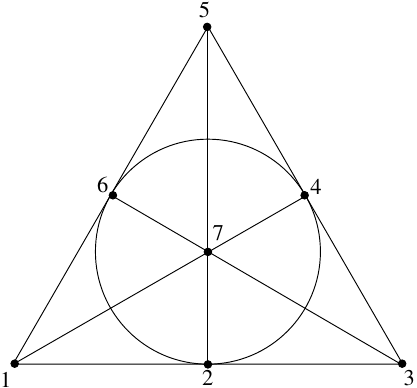}
\caption{The Fano matroid}\label{fig:fano}
\end{figure}
\begin{figure}
\centering
\includegraphics[width=14cm]{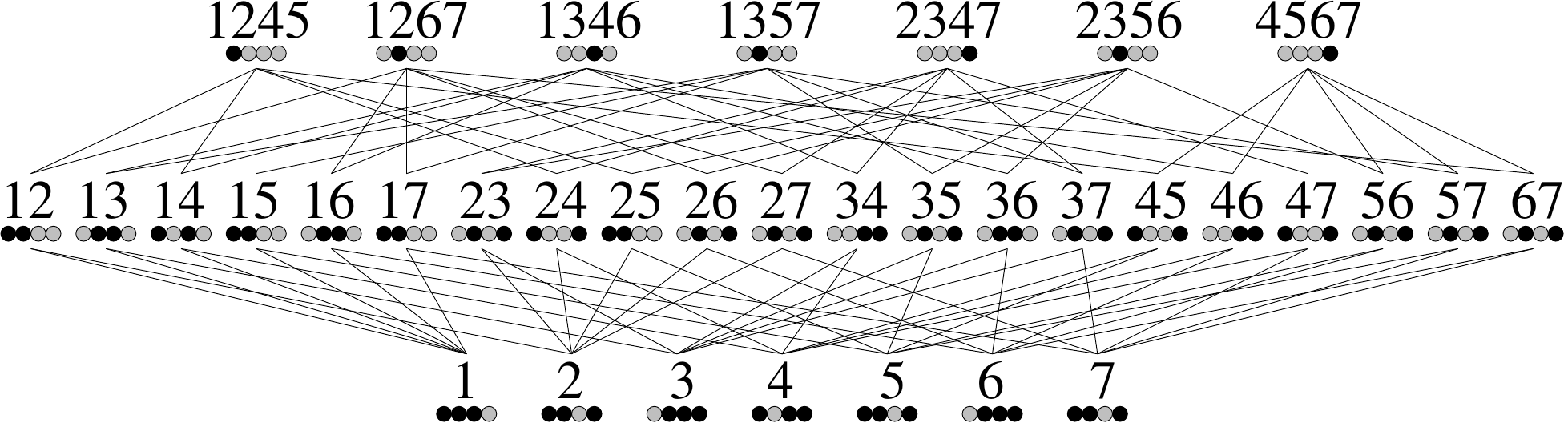}
\caption{A diagram of spaces on the Fano matroid}\label{fig:fanoLattice}
\end{figure}
\begin{figure}
\centering
\includegraphics{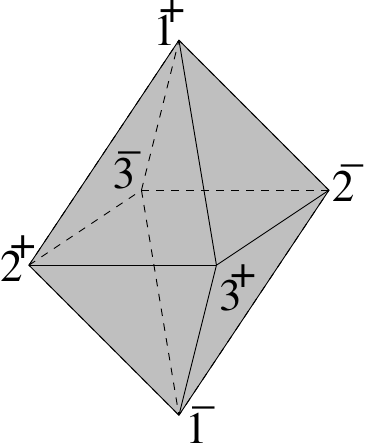}
\caption{The space $\mathcal{D}_1$.}\label{fig:fanoCyl1}
\end{figure}
\begin{figure}
\centering
\includegraphics{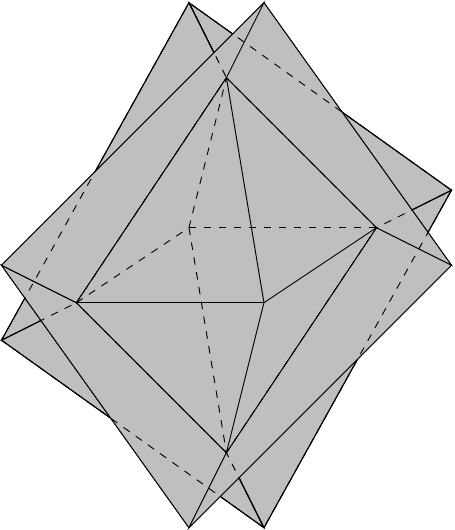}
\caption{The mapping cylinders of $\mathcal{D}_{12} \rightarrow \mathcal{D}_1$ and  $\mathcal{D}_{15} \rightarrow \mathcal{D}_1$ .}\label{fig:fanoCyl2}
\end{figure}
\begin{figure}
\centering
\includegraphics{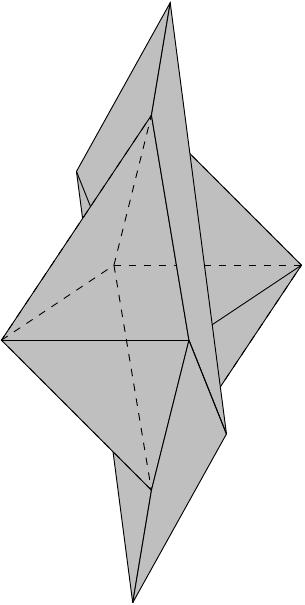}
\caption{The mapping cylinder of $\mathcal{D}_{14} \rightarrow \mathcal{D}_1$.}\label{fig:fanoCyl3}
\end{figure}
\begin{figure}
\centering
\includegraphics{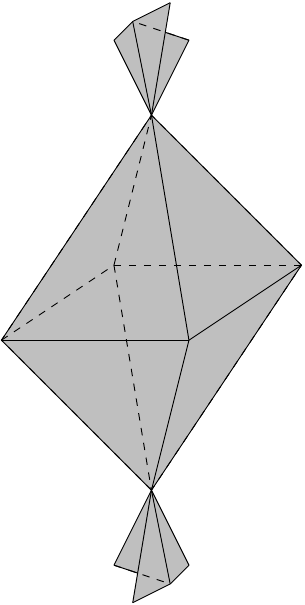}
\caption{The mapping cylinders of $\mathcal{D}_{1245} \rightarrow \mathcal{D}_{12} \rightarrow \mathcal{D}_1$,  $\mathcal{D}_{1245} \rightarrow \mathcal{D}_{14} \rightarrow \mathcal{D}_1$ and $\mathcal{D}_{1245} \rightarrow \mathcal{D}_{15} \rightarrow \mathcal{D}_1$.}\label{fig:fanoCyl4}
\end{figure}

\end{document}